\documentclass[a4paper,11pt]{amsart}

\usepackage{cite}
\usepackage{mathtools}
\usepackage{algpseudocode}
\usepackage{amsmath,amsthm}
\usepackage{textcomp}
\usepackage{amssymb}
\usepackage[all]{xy}
\usepackage{enumerate}
\usepackage{fancyhdr}
\usepackage{lscape}
\usepackage{url}
\usepackage{setspace}
\usepackage{longtable}
\usepackage{tikz}
\usepackage{wasysym}

\usepackage{wasysym}

\theoremstyle{plain}
\newtheorem{theorem}{Theorem}[section]

\theoremstyle{definition}

\newtheorem{lm}[theorem]{Lemma}

\newtheorem{prop}[theorem]{Proposition}
\newtheorem{cor}[theorem]{Corollary}
\newtheorem{conj}[theorem]{Conjecture}

\newtheorem{question}[theorem]{Question}

\theoremstyle{remark}
\newtheorem{rmk}{Remark}

\DeclareMathOperator{\dP}{dP}

\DeclareMathOperator{\Int}{Int}

\DeclareMathOperator{\Mob}{Mob}

\newcommand{\Cr}{\operatorname{Cr}}
\newcommand{\Sing}{\operatorname{Sing}}

\newcommand{\Aut}{\operatorname{Aut}}

\newcommand{\Pic}{\operatorname{Pic}}

\newcommand{\psl}{\operatorname{PSL}_2(\mathbb{F}_7)}
\newcommand{\sll}{\operatorname{SL}_2(\mathbb{F}_8)}
\newcommand{\psp}{\operatorname{PSp}_4(\mathbb{F}_3)}

\def \O {\mathcal{O}}

\def \X {\mathcal{X}}

\def \PPP{\mathcal{P}}

\def \C {\mathbb{C}}
\def \Z {\mathbb{Z}}

\def \A {\mathbb{A}}

\newcommand{\PP}{\mathbb{P}}
\renewcommand{\H}{\text{H}}

\newcommand\qt{{\slash\kern-0.65ex\slash}}

\title[Klein simple group in Cremona group]{On conjugacy classes of the Klein\\simple group in Cremona group}

\author{Hamid Ahmadinezhad}

\keywords{Birational Geometry; Cremona Group.\vspace{0.2cm} }
\subjclass[2010]{14E07}

\begin{document}

\begin{abstract}We consider countably many three dimensional $\psl$-del Pezzo surface fibrations over $\PP^1$. Conjecturally they are all irrational except two families, one of which is the product of a del Pezzo surface with $\PP^1$. We show that the other model is $\psl$-equivariantly birational to $\PP^2\times\PP^1$. Based on a result of Prokhorov, we show that they are non-conjugate as subgroups of the Cremona group $\Cr_3(\C)$. \end{abstract}

\maketitle


\section{Introduction}

The Klein simple group $\psl$ appears in various branches of mathematics.\begin{enumerate}[-]
\item{\bf In Group Theory}: it is the maximal subgroup of the Mathieu group and it is the second smallest non-abelian simple group.
\item{\bf In Hyperbolic Geometry}: it is the automorphism group of the Klein quartic $$\{x^3y+y^3z+z^3x=0\}\subset\PP^2.$$
\item{\bf In Finite Geometry}: it is the symmetry group of the Fano plane.
\item{\bf In Algebraic Geometry}: it is one of the three, respectively six, finite simple non-abelian groups that admit an embedding in the two dimensional, respectively three dimensional, Cremona group. 
\end{enumerate}

The latter motivates this paper. For simplicity we denote the Cremona group of rank~$n$, the group of birational automorphisms of the complex projective space $\PP^n$, by $\Cr_n(\C)$. Suppose $X$ is an $n$-dimensional rational variety, that is a variety birational to $\PP^n$, and let $G$ be a subgroup of $\Aut(X)$. Then the rational map $\varphi\colon X\dasharrow\PP^n$ defines an embedding of $G$ into $\Cr_n(\C)$. If $G$ acts on two rational varieties $X$ and $Y$ such that there does not exist a $G$-equivariant birational map $\psi\colon X\dashrightarrow Y$ then the two embeddings of $G$ in $\Cr_n(\C)$ cannot be conjugate.

The following question of Serre brought special attention to subgroups of $\Cr_3(\C)$:

\begin{question}[Serre\cite{serre}]  Does there exist a finite group which cannot be embedded in $\Cr_3(\C)$?
\end{question}

In \cite{prokh}, Prokhorov gave a negative answer to this by showing that there are only six finite simple non-abelian groups that admit an embedding into $\Cr_3(\C)$, namely \[\mathbb{A}_5, \mathbb{A}_6, \mathbb{A}_7, \psl, \sll \text{ and }\psp. \]

It is a rule of thumb that {\it the larger the group, the fewer non-conjugate embeddings it admits}. For instance all non-conjugate embeddings of $\mathbb{A}_7$, $\sll$ and $\psp$ are known \cite{beauv, five-emb,prokh}. On the other hand $\mathbb{A}_6$ admits at least five non-conjugate embeddings in $\Cr_3(\C)$, see \cite{five-emb}.

\subsection*{Embeddings of $\psl$ in $\Cr_2(\C)$} 

It is known that $\psl$ admits two embeddings into $\Cr_2(\C)$. One embedding is obtained by the direct action of $\psl$ on $\PP^2$ that leaves the Klein quartic curve invariant. The other embedding is obtained by action of $\psl$ on a double cover of $\PP^2$ branched over the Klein quartic. This double cover is known as the del Pezzo surface of degree $2$, which is the blow up of $\PP^2$ at $7$ points in general position (denoted by $\dP_2$). Both these varieties turn out to be $\psl$-birationally rigid, in particular $\PP^2$ is not $\psl$-equivariantly birational to $\dP_2$, hence the two embeddings are non-conjugate (see \cite{cheltsov-ineq} and Appendix B therein).

Section\,\ref{stable-conj} reviews the stable non-conjugacy of $\psl$ in $\Cr_2(\C)$ based on the work of Prokhorov and Bogomolov. Two embeddings of a group $G\subset\Cr_n$ obtained from the action on $X$ and $Y$ are said to be stably conjugate if there exists a $G$-equivariant birational map $\Phi\colon X\times\PP^m\dashrightarrow Y\times\PP^m$, for some $m\geq 1$, where the action on the base is trivial.

In Section\,\ref{Cr3} I discuss embeddings of $\psl$ into $\Cr_3(\C)$. The case when the embedding is obtained from a $\psl$-Fano 3-fold is known \cite{three-emb}. I construct infinitely many $\psl$-del Pezzo fibrations and argue that they are all (conjecturally) irrational, and perhaps $\psl$-birationally rigid, except two families, which are both rational. It is a conjecture (Cheltsov and Shramov) that these are the only $\psl$-del Pezzo fibrations in dimension three \cite[Conjecture\,1.4]{belousov}. I show that one of these two families is $\psl$-equivariantly birational to $\PP^2\times\PP^1$, which gets a step closer to this conjecture.

\paragraph*{\bf Acknowledgement} I would like to thank J\'anos Koll\'ar and Konstantin Shramov for showing interest in this problem. I am grateful to the anonymous referee for their constructive feedback.

\section{Stable non-conjugacy of $\psl$ in $\Cr_2(\C)$}\label{stable-conj}
In this section I review the non-conjugacy of $\psl$ in $\Cr_n(\C)$ for $n\geq 2$ coming from the action of $\psl$ on $\PP^2$ and $\dP_2$. All results in this section can essentially be found in \cite{bogo-prokh} and  \cite{yuri}.

The proof of non-conjugacy in $\Cr_2(\C)$ is based on $\psl$-birational rigidity of both varieties $\PP^2$ and $\dP_2$, implying that for any of these two varieties there is no $\psl$-equivariant birational map to a Mori fibre space other than itself. See \cite{pukhlikov} for an introduction to birational rigidity. However, while given current tools in hand it is nearly impossible to prove any stable birational rigidity statements, whatever that notion means, it is not even true that $\PP^2\times\PP^1$ is $\psl$-birationally rigid. This is shown in Lemma\,\ref{lemma:Y-vs-X1}.

On the other hand, in \cite{bogo-prokh} Bogomolov and Prokhorov introduced a stable conjugacy invariant, that is $\H^1\big(G,\Pic(X)\big)$. One can verify that
\[\H^1\big(\psl,\Pic(\PP^2)\big)=\H^1\big(\psl,\Pic(\dP_2)\big)=0,\]
which implies this invariant will not solve our problem. Instead, one should look at the subgroups of $G$ together with this invariant (Lemma\,\ref{collection} below). 

 Let $X$ be a variety on which the group $G$ acts biregularly. Define the collection of group cohomologies associated to $X$ and subgroups of $G$ by
\[\wp(X,G)=\big\{ \H^1\big(H, \Pic(X)\big)\text{, where } H\subset G\subset\Aut(X)\big\}.\]

\begin{lm}\label{collection} Let $G$ act biregularly on two varieties $X$ and $Y$. If $X$ and $Y$ are $G$-stably birational then $\wp(X,G)=\wp(Y,G)$.
\end{lm}
\begin{proof}Let $H$ be a subgroup of $G$ acting on $X$ and $Y$. It follows from a standard proof that $\H^1\big(H, \Pic(X)\big)\cong\H^1\big(H, \Pic(Y)\big)$. See \cite[Proposition\,2.2]{bogo-prokh} and references therein.
\end{proof}

\begin{theorem}  The two embeddings of $\psl$ in $\Cr_2(\C)$ are not stably conjugate.
\end{theorem}
\begin{proof} This is an immediate consequence of Theorem\,1.2 in\,\cite{yuri} and Lemma\,\ref{collection} above. Note that the $\Z_2\subset\psl$ fixes a line in $\PP^2$, where the pre-image of this line in $\dP_2\xrightarrow{2:1}\PP^2$, is an elliptic curve.
\end{proof}

\begin{cor} For any $n\geq 2$ there are at least two non-conjugate embeddings of $\psl$ in $\Cr_n(\C)$.
\end{cor}

\section{Embeddings of $\psl$ in $\Cr_3(\C)$}\label{Cr3}

It is a standard technique to construct embeddings of a group $G$ in $\Cr_n(\C)$ by constructing rational Mori fibre spaces which admit a biregular
$G$-action. Mori fibre spaces in dimension three are either Fanos or del Pezzo fibrations over a curve or conic bundles over surfaces.

In \cite{three-emb} Cheltsov and Shramov constructed three non-conjugate embeddings of $\psl$ in $\Cr_3(\C)$ using the action of this group on $\PP^3$ and a special Fano variety in the famous Fano family $V_{22}$. The two families considered in Section\,\ref{stable-conj} are del Pezzo fibrations. In this section I construct infinitely many families of del Pezzo fibrations, one of which is the $\dP_2\times\PP^1$, that admit an action of $\psl$. It should be mentioned that these families have been considered by Belousov \cite{belousov} with a different description. Then I show that a member of this family is $\psl$-birational to $\PP^2\times\PP^1$, which makes it $\psl$-birationally nonrigid. Then I conjecture that these are all all irrational except the two particular ones, hence claiming that there are only two embeddings of $\psl$ in $\Cr_3(\C)$ obtained from del Pezzo fibrations.

\subsection*{$\psl$-del Pezzo fibrations over $\PP^1$}

Let $\PP=\PP^1\times\PP(1,1,1,2)$,
and denote by $\pi$, the natural projection $\pi\colon\PP\to\PP^1$. 
Suppose
\[f(x,y,z)=x^3y+y^3z+z^3x\]
and let 
$\X'_n\subset\PP$, for a non-negative integer $n$, be a $3$-fold given by the equation 
\[\alpha_n(u,v)t^2+\beta_n(u,v)f(x,y,z)=0\,, \text{ where}\]
\begin{enumerate}[(i)]
\item $u$ and $v$ are the homogeneous coordinates on $\PP^1$,
\item $x$, $y$ and $z$ are weighted homogeneous coordinates 
of weight $1$ on $\PP(1,1,1,2)$, and 
$t$ is a weighted homogeneous coordinate
of weight $2$ on $\PP(1,1,1,2)$,
\item $\alpha_n$ and $\beta_n$ are general homogeneous polynomials of degree $n$ such that  $|Z_\alpha|=|Z_\beta|=n$ and $Z_\alpha\cap Z_\beta=\varnothing$, where $Z_\alpha=\{\alpha_n=0\}\subset\PP^1$ and $Z_\beta=\{\beta_n=0\}\subset\PP^1$.
\end{enumerate}
There is an action of the group $\psl$ on $\X_n'$, induced from the natural action of $\psl$ on the fibres ($\dP_2$ surfaces) of the projection $\X_n'\rightarrow\PP^1$.

Note that the variety $\X'_0$ is nothing but $\dP_2\times\PP^1$.
The variety $\X'_1$ is unique, since there is only one pair 
of linear forms with distinct zeroes on $\PP^1$ up to a change of coordinates.

Let $T\subset\X'_n$ be the divisor defined by the equation $f(x,y,z)=0$.
Denote by $a_1,\dots, a_n\in\PP^1$ the 
points of the set $Z_\beta$, and by $S_1, \dots, S_n$ the fibres 
of the induced fibration
$\pi\colon\X'_n\to\PP^1$ over these points, and let $C_i=S_i\cap T$.
Let $p_1,\dots,p_n$ be the points given by  $\{\alpha=x=y=z=0\}\subset\X'_n\,$, and denote by $S'_1, \ldots, S'_n$ the fibres
of $\pi$ passing through these points.

The following lemma follows from the construction of $\X'_n$.
\begin{lm}
For $\X'_n$, constructed as above, we have
\begin{enumerate}[(i)]
\item $Sing(\X'_n)=\big\{p_1,\dots, p_n\big\}\cup C_1\cup\dots\cup C_n$,
\item each of the points $p_j$ is a singular point of type $\frac{1}{2}(1,1,1)$ on $\X'_n$, and $\X'_n$ is locally isomorphic to $\A_1\times\C$ along the curves $C_i$,
\item the fibres $S_1, \dots, S_n$ are non-reduced fibres of $\pi$,
and $S_i\cong\PP^2$,
\item each of the fibres $S^\prime_1,\dots,S^\prime_n$ has a unique singularity at the point 
$p_j$, and is isomorphic to the cone
over the plane quartic curve $f(x,y,z)=0$.
\item if $S$ is a fibre of $\pi$ different from all $S_i$ and $S_j'$, then 
$S$ is non-singular.
\end{enumerate}
\end{lm}

Let $\nu\colon \tilde{\X}_n\to\X'_n$ be a blow up 
of the $3$-fold $\X'_n$ at the curves
$C_1,\ldots, C_n$, and $\mu\colon \tilde{\X}_n\to\X_n$ be a 
contraction of the strict transforms $\tilde{S}_i$ of the non-reduced 
fibres $S_i$ on $\tilde{\X}_n$.
Both $\nu$ and $\mu$ are $\psl$-equivariant birational
morphisms:
\[
\xymatrix{
&\tilde{\X}_n\ar@{->}[dr]^{\mu}\ar@{->}[ld]_{\nu}&\\%
\X'_n&&\X_n}
\]

Let $Q_i=\mu(\tilde{S}_i)$ for $1\le i\le n$.
Since $\mu\circ\nu^{-1}$ is an isomorphism in the neighbourhood of the 
points $P_j\in\X'_n$, I will use the same letters $P_j$ to denote the 
corresponding points of $\X_n$.

\begin{prop}\label{Xn-properties} For $\X_n$ as above, we have
\begin{enumerate}[(i)]
\item $\Sing(\X_n)=\big\{P_1,\dots, P_n, Q_1, \dots, Q_n\big\}$,
\item each of the points $P_j$ and $Q_i$ is a
singular point of type $\frac{1}{2}(1,1,1)$ on $\X_n$, and in particular
\item the variety $\X_n$ is a $\psl$-Mori fibre space.
\end{enumerate}
\end{prop}

\begin{proof} In fact $\X_n$ can be described as a hypersurface in a toric variety $T$ of Picard number two. The coordinate ring of the toric space is a $\Z^2$-graded ring with variables $u,v,x,y,z,t$, and grading $(1,0), (1,0)$ for $u$ and $v$, and $(0,1), (0,1),(0,1),(-n,2)$ for $x,y,z,t$, with irrelevant ideal $(u,v)\cap(x,y,z,t)$. Then $X_n$ is a degree $(0,4)$ hypersurface defined by $\alpha\beta t^2+f=0$. The singular locus of $T$ is $\PP^1_{u:v}\times\frac{1}{2}(1,1,1)$ quotient singularity. This locus is cut out by $\X_n$ in $2n$ points, the solutions of $\alpha\beta=0$ in $\PP^1$, so that $\X_n$ has $2n$ singular points of type $\frac{1}{2}(1,1,1)$. Clearly $\X_n$ is singular, as it has only isolated quotient singularities, and its Picard number is $2$ by the Lefschetz property. Hence, the fibration to $\PP^1$ is a $\dP_2$ fibration and a Mori fibration.
\end{proof}

\begin{rmk}Note that with the description in the proof of Proposition\,\ref{Xn-properties} the birational map between $\X'_n$ and $\X_n$ can be recovered from $t\longleftrightarrow \beta t$, between $T\longleftrightarrow \PP^1\times\PP(1,1,1,2)$.\end{rmk}

\begin{conj}[Cheltsov-Shramov, Conjecture\,1.4\,\cite{belousov}]
The varieties $\PP^2\times\PP^1$ and $\X_n$, for $n\ge 0$, are the only 
$\psl$-Mori fibre spaces over $\PP^1$ in dimension $3$.
\end{conj}

\begin{theorem}\label{lemma:Y-vs-X1}
There is a $\psl$-equivariant birational equivalence 
between the varieties $\PP^2\times\PP^1$ and $\X_1$.
\end{theorem}
\begin{proof}
It is immediate to see that $\X_1'$ (and thus also $\X_1$)
is $\psl$-equivariantly birational to $\PP(1,1,1,2)$, by projection.
I now explain how to get from $\PP(1,1,1,2)$ equivariantly to $\PP^1\times\PP^2$.

By blowing up the singular point of $\PP(1,1,1,2)$,
we obtain a $\psl$-equivariant birational 
morphism 
$$\PP_{\PP^2}\big(\O_{\PP^2}\oplus\O_{\PP^2}(2)\big)\to\PP(1,1,1,2).$$

Note that for any $n\in\Z$ there is an action of $\psl$ on the 
projectivization 
$$\PPP_n\cong \PP_{\PP^2}\big(\O_{\PP^2}\oplus\O_{\PP^2}(n)\big)$$
arising from the action of $\psl$ on the base $\PP^2$.

Let use define a typical fibrewise transform on $\PPP_n$. Suppose that $\Sigma_0$ is a $\psl$-invariant section of the projection $\pi_n\colon\PPP_n\to\PP^2$. Let $C\subset\Sigma_0\cong\PP^2$ be a smooth $\psl$-invariant 
curve of degree $d$, defined by $\{g=0\}$. By blowing up $\PPP$ along $C$ one has a diagram of $\psl$-equivariant morphisms
$$
\xymatrix{
&\PPP_n^0\ar@{->}[dr]^{\beta_0'}\ar@{->}[ld]_{\beta_0}&\\%
\PPP_n&&\PPP_{n+d}}
$$
where $\beta_0$ is a blow up of the curve $C$, and $\beta_0'$ 
is a contraction of the strict transform of the divisor 
$\pi_n^{-1}\big(\pi_n(C)\big)$. This map can be seen in coordinates as
\[(x:y:z;a:b)\longmapsto (x:y:z;a:gb),\]
where $x:y:z$ are the coordinates on the $\PP^2$ and $a:b$ the coordinates on the fibre. Clearly this map is $\psl$-equivariant, and shows that, given an invariant curve $C:\{g=0\}$ of degree $d$, one can birationally move between $\PPP_n$ and $\PPP_{n-d}$.

Using this, in our situation one can use the Klein quartic curve $C$ to show that $\PP_2$ is $\psl$-equivariantly birational to $\PPP_6$. On the other hand, we showed earlier that $\X_1$ is $\psl$-equivariantly birational to $\PPP_2$.

Now, let $\Sigma_{\infty}\subset\PPP_6$ be a $\psl$-equivariant sections of $\pi_6$, and 
let $C_6\subset\Sigma_{\infty}$ be the Hessian curve of the Klein
quartic. Then $C_6$ is a $\psl$-invariant curve of degree $6$,
hence provides a $\psl$-equivariant
birational map $\PPP_6\dasharrow\PPP_0$.

Combining the birational maps obtained 
above, we get a sequence of $\psl$-equivariant
birational maps
$$\X_1\dasharrow\X_1'\dasharrow\PP(1,1,1,2)\dasharrow\PPP_2\dasharrow
\PPP_6\dasharrow\PPP_0\cong\PP^1\times\PP^2.$$
\end{proof}

Inspired by the work of Grinenko \cite{grinenko} a natural expectation in birational geometry arises: a 3-fold del Pezzo fibration $X$ of degree $2$ (or $1$) with only quotient singularities is birationally rigid if and only if $-K_X\notin\Int\Mob(X)$. For a discussion on this, and a counterexample in case we allow other singularities, I refer to \cite{singular-dP} and the references therein. This expectation translates to the following conjecture in our case.

\begin{conj}\label{conjecture:PSL}
The varieties $\X_n$ are birationally rigid, and in particular irrational, for $n\ge 2$.
\end{conj}


\vspace{0.2cm}

School of Mathematics, University of Bristol, Bristol, BS8 1TW, UK

e-mail: \url{h.ahmadinezhad@bristol.ac.uk}

\end{document}